\documentclass[12pt]{amsart}
\usepackage{graphicx,amsmath,mathtools}
\usepackage{epsfig}
\usepackage{fancyhdr}
\newtheorem{theorem}{Theorem}[section]
\newtheorem{lemma}[theorem]{Lemma}
\newtheorem{definition}[theorem]{Definition}
\newtheorem{example}[theorem]{Example}
\newtheorem{proposition}[theorem]{Proposition}
\newtheorem{corollary}[theorem]{Corollary}

\newtheorem{problem}[theorem]{Problem}
\numberwithin{equation}{section}

\begin{document}

\title{Strict unimodality of $q$-polynomials of rooted trees}

\author{Zhiyun Cheng}
\address{School of Mathematical Sciences, Laboratory of Mathematics and Complex Systems, Beijing Normal University, Beijing 100875, China
\newline
\indent and
\newline
\indent Department of Mathematics, The George Washington University, Washington, DC 20052, U.S.A.}%
\email{czy@bnu.edu.cn}
%\thanks{The first author was supported by the NSFC 11301028.}

\author{Sujoy Mukherjee}
\address{Department of Mathematics, The George Washington University, Washington DC, USA}
\email{sujoymukherjee@gwu.edu}

\author{J\'{o}zef H. Przytycki}
\address{Department of Mathematics, The George Washington University, Washington DC, USA, and University of Gda\'nsk}
\email{przytyck@gwu.edu}
%\thanks{The third author was supported in part by the Simons Collaboration Grant-316446.}

\author{Xiao Wang}
\address{Department of Mathematics, The George Washington University, Washington DC, USA}
\email{wangxiao@gwu.edu}

\author{Seung Yeop Yang}
\address{Department of Mathematics, The George Washington University, Washington DC, USA}
\email{syyang@gwu.edu}

\date{January 13, 2016}
% and, in revised form, ????, 2016.}
%Mathathon I, Dec 14-21, 2015

\begin{abstract}
We classify rooted trees which have strictly unimodal $q$-polynomials (plucking polynomial).
We also give criteria for a trapezoidal shape of a plucking polynomial.
We generalize results of Pak and Panova on strict unimodality of $q$-binomial coefficients.
We discuss which polynomials can be realized as plucking polynomials and whether or not different rooted trees can have the same plucking polynomial.
\end{abstract}

\maketitle
%\markboth{\hfil{\sc Strictly Unimodal}\hfil}

\tableofcontents
\section{Introduction}

We study in this paper properties of coefficients of the $q$-polynomial invariant of rooted trees. This invariant is defined, initially,
for plane rooted trees using the recursive plucking recursive relation as follows \cite{Prz-1,Prz-2,Prz-3}.
In our work, we use the convention that trees are growing up (like in Figure $1.1$).
\begin{definition}\label{Definition 1.1}
Consider the plane rooted tree $T$ (compare  Figure $1.1$). We associate with $T$ the polynomial
$Q(T,q)$ (or succinctly $Q(T)$) in the variable $q$ as follows.
\begin{enumerate}
\item[(i)] If $T$  is the one vertex tree, then $Q(T,q)=1$.
\item[(ii)] If $T$ has some edges (i.e. $|E(T)|>0$), then
$$Q(T,q)=\sum_{v \in \mbox{ leaves }}q^{r(T,v)}Q(T-v,q), $$
where the sum is taken over all leaves, i.e. vertices of degree $1$ (not a root), of $T$
 and $r(T,v)$ is the number of edges of $T$ to the right of the unique path connecting $v$ with the root
 (examples are given in Figure $1.1$).
\end{enumerate}
\end{definition}
We call the polynomial $Q(T)$ the {\it plucking polynomial} of $T$ motivated by the nature of the definition.

\begin{figure}[h]
\centerline{\psfig{figure=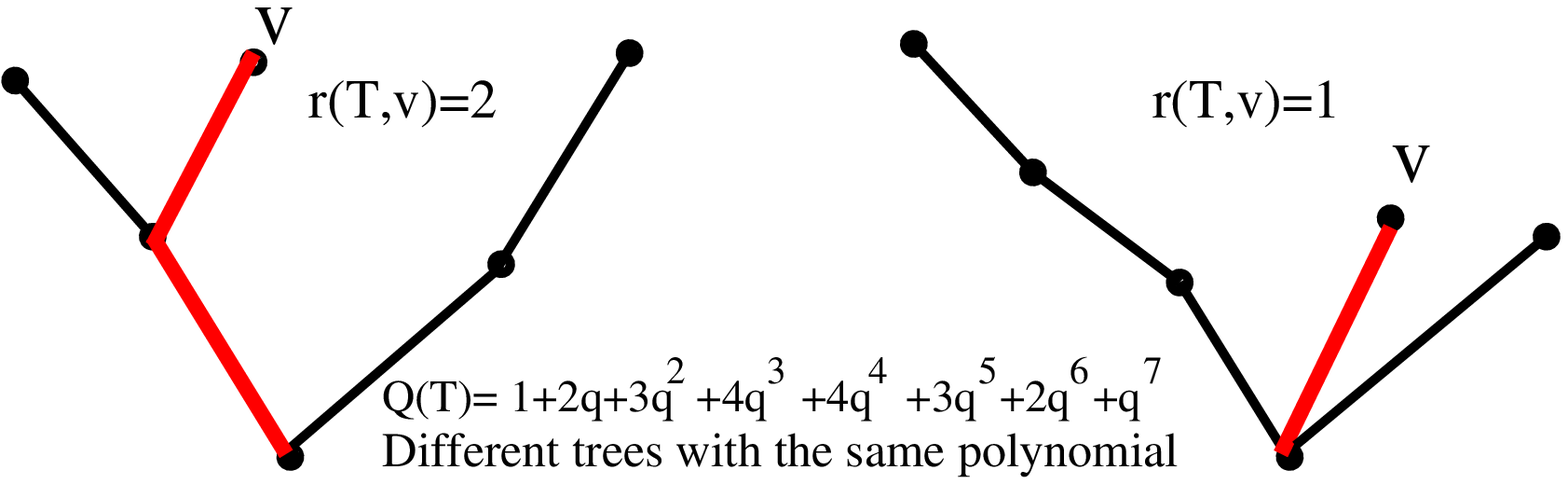,height=3.3cm}}
\ \\ \ \\
\centerline{ Figure $1.1$: Plane rooted trees and examples of $r(T,v)$}
\end{figure}

We proved in \cite{Prz-1,Prz-3} that $Q(T)$ is a rooted tree invariant (so it does not depend on the embedding in the plane).
It follows from the following result that we shall often use in the paper.

\begin{theorem}\cite{Prz-1,Prz-3}\label{Theorem 1.2}\
\begin{enumerate}
\item[(1)]
Let $T_1 \vee T_2$ be a wedge product of trees $T_1$ and $T_2$ ({\psfig{figure=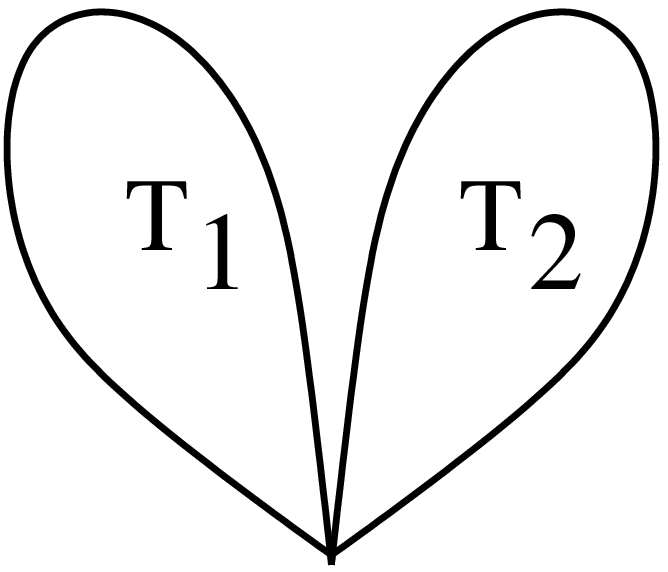,height=1.1cm}}). Then:
$$Q(T_1 \vee T_2)= \binom{|E(T_1)|+|E(T_2)|}{|E(T_1)|,|E(T_2)|}_qQ(T_1)Q(T_2)$$
\item[(2)]  Let a plane tree be a wedge product of $k$ trees (\psfig{figure=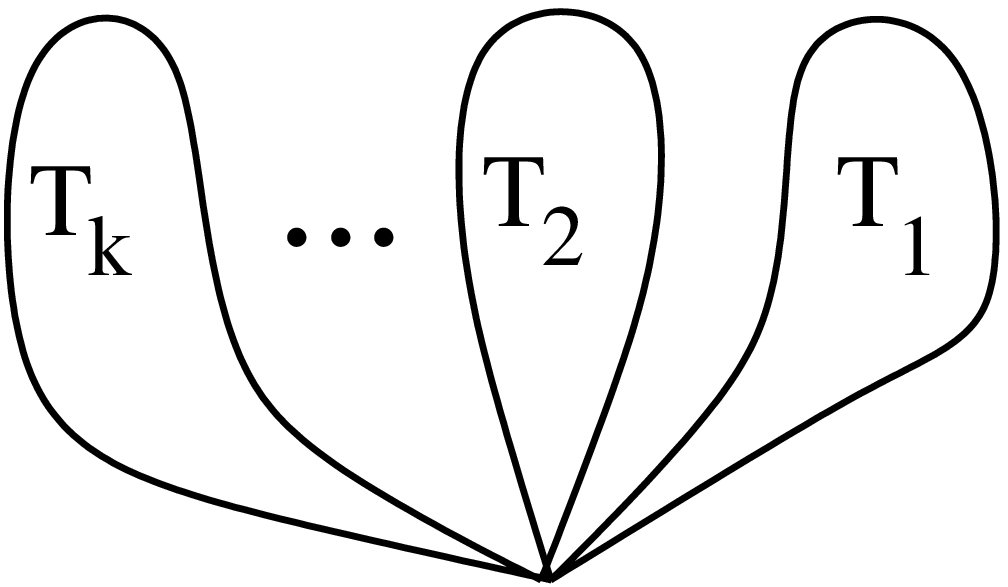,height=1.5cm}), that is
$T=T_{k} \vee \cdots \vee  T_2 \vee T_1,  \mbox{ then }$
$$Q(T)= \binom{E_k+E_{k-1}+...+E_1}{E_k,E_{k-1},...,E_1}_qQ(T_k)Q(T_{k-1})\cdots Q(T_1),$$
where $E_i=|E(T_i)|$ is the number of edges in $T_i$, and $q$-multinomial coefficients are defined by
$$\binom{E_k+E_{k-1}+...+E_1}{E_k,E_{k-1},...,E_1}_q=\frac{[E_k+...+E_2+E_1]_q!}{[E_k]_q!\cdots [E_2]_q!\ [E_1]_q!}.$$
Here $[k]_q=1+q+...+q^{k-1}$ and $[k]_q!= [k]_q[k-1]_q \cdots[2]_q[1]_q$ are called a $q$-integer and a $q$-factorial, respectively.\\
Notice that for every ordering of numbers $E_1,E_2,...,E_k$ we can decompose the $q$-multinomial coefficient
$\binom{E_k+E_{k-1}+...+E_1}{E_k,E_{k-1},...,E_1}_q$ into the product of Gaussian polynomials:
$${E_k+E_{k-1}+...+E_1 \choose E_k,E_{k-1}+...+E_k}_q \cdots{E_3+E_2+E_1\choose E_3,E_2+E_1}_q{E_2+E_1\choose E_2,E_1}_q.$$
\item[(3)] (State product formula)
$$ Q(T) = \prod_{v\in V(T)}W(v), $$
where $W(v)$ is the weight of a vertex (we can call it the Boltzmann weight) defined by:
$$W(v)= \binom{E(T^v)}{E(T^v_{k_v}),...,E(T^v_{1})}_q,$$
where $T^v$ is a subtree of $T$ with root $v$ (part of $T$ above $v$, in other words $T^v$ grows from $v$)
and $T^v$ may be decomposed into a wedge of trees as follows: $T^v= T^v_{k_v} \vee \cdots \vee  T^v_2 \vee T^v_{1}.$\\
Notice that by $(2)$, $Q(T)$ is a product of $q$-binomial coefficients.
\item[(4)] $Q(T)$ is of the form $c_0+c_1q+...+c_Nq^N$ where:\\
(i) $c_0=1=c_N$, $c_i>0$ for every $i\leq N$, \\
(ii) $c_i=c_{N-i}$ for each $i$, that is  $Q(T)$ is a symmetric polynomial (i.e. a palindromic polynomial),\\
(iii) the sequence $c_0,c_1,...,c_N$ is unimodal (see the next definition),\\
(iv) For a nontrivial tree $T$, that is a tree with at least one edge, we have:
$$c_1 = \sum_{v\in V(T)} (k_v-1),$$ where $k_v= deg_{T^v}(v)$ is the number of edges growing up from $v$, that is the degree of
$v$ in the tree $T^v$ growing from $v$. The number $c_1$ of $Q(T)$ is called the branching number of $T$.
\end{enumerate}
\end{theorem}
Notice that for trees of Figure $1.1$ we get $Q(T)={5 \choose 2,3}_q[3]_q = [4]_q[5]_q$, and that $c_1=2$.
\begin{definition}\label{Definition 1.3}
 Consider the sequence of nonnegative integers $c_0,c_1,...,c_N$.
\begin{enumerate}
\item[(1)] If there is $j$ ($0\leq j \leq N$) such that $c_0\leq...\leq c_{j-1} \leq c_j \geq c_{j+1} \geq \ldots \geq c_N$
and $c_0>0, c_N>0,$ then we call the sequence {\it unimodal} of length $N$.
\item[(2)] If $c_i=c_{N-i}$ for every $i,$ then we call the sequence {\it symmetric} (or palindromic) centered at $N/2$.
\end{enumerate}
In this paper, we deal exclusively with unimodal symmetric sequences with $c_0=c_N=1$.
\begin{enumerate}
\item[(3)] If additionally $c_0<c_1< ...< c_{\lfloor N/2 \rfloor}= c_{\lceil N/2 \rceil} >...>c_{N-1}>c_N,$ then we call the sequence a {\it strictly
unimodal, symmetric sequence centered at $N/2$}. Here $\lfloor x \rfloor$ and $\lceil x \rceil$ denote the floor (Entier) and the ceiling
of the number $x$, respectively.
\item[(4)] If we do not require $c_0<c_1$ and $c_{N-1}>c_{N}$ but still assume $c_1<c_2< ...<c_{\lfloor N/2 \rfloor}= c_{\lceil N/2 \rceil} >...>c_{N-1}$
in $(3)$, we call the sequence {\it almost strictly unimodal}.
\item[(5)] If a symmetric unimodal sequence satisfies for some $j$:
$$c_0<  c_1 <...<c_j=...=c_{N-j}>...>c_{N-1}>c_N,$$ then the sequence is called a {\it trapezoidal} sequence with base of length $N$
and top of length $N-2j$ (i.e. $N-2j+1$ terms).
\item[(6)] If we assume only $c_1 <...<c_j=...=c_{N-j}>...>c_{N-1}$ in $(5),$ then we say we have an {\it almost trapezoidal} sequence with a top of length $N-2j$.
\item[(7)] We say that a polynomial in one variable $q$ with nonnegative coefficients is symmetric unimodal (respectively,
almost unimodal, trapezoidal, or almost trapezoidal) if its nonzero coefficients form a symmetric unimodal sequence (respectively,
almost unimodal, trapezoidal, or almost trapezoidal).
\item[(8)] We denote by $PSU_N$ the set of positive, symmetric, unimodal polynomials of degree $N$.
\end{enumerate}
\end{definition}
The classical result of Sylvester (1878, \cite{Syl}) established the unimodality of Gaussian polynomials ($q$-binomial coefficients):
$${m+n \choose m,n}_q = \frac{[m+n]_q!}{[m]_q! [n]_q!},$$
For us the following observation of MacMahon is of importance:
$$\mbox{If $yx=qxy$,  then } (x+y)^N= \sum_{m+n=N}{m+n \choose m,n}_q x^my^n.$$
For more basic information on Gaussian polynomials ($q$-binomial coefficients) we refer to \cite{K-C}.
In particular, Gaussian polynomials are symmetric centered at $\frac{mn}{2}$ and of degree $mn$.

Our starting point is the result by Pak and Panova \cite{Pak-Pan} describing almost strict unimodality of Gaussian polynomials
and listing the exceptions from almost unimodality.

\begin{theorem}(\cite{Pak-Pan})\label{Theorem 1.4} For all $m,n \geq 8$, the Gaussian polynomials ${m+n \choose m,n}_q$ are almost strictly unimodal.
Furthermore, for $m,n \geq 5$ the Gaussian polynomials are almost strictly unimodal with nine exceptions (we assume $m\leq n$):
$$\{(5,6), (5,10), (5,14), (6,6) , (6,7), (6,9), (6,11), (6,13),  (7,10)\}.$$

Furthermore, in all exceptional cases but $(6,6)$\footnote{In \cite{Pak-Pan}, the case of $(6,6)$ is described also, by mistake, as one
with ``the middle three coefficients equal."}
 the Gaussian polynomial is almost trapezoidal with a top of length $2$ ($3$ terms). The explicit formula of $(6,6)$ is the following:
$$ {12 \choose 6,6}_q =$$
$$1+ q +  2q^2 + 3 q^3 + 5 q^4 + 7 q^5 + 11 q^6 +...+$$
$$51 q^{15} + 55 q^{16} + 55 q^{17} + 58 q^{18} + 55 q^{19} + 55 q^{20} + 51 q^{21} + ...+q^{36}.$$
In such a case, we say that the top of the polynomial has a shape of type $(2,1,2)$; see Figure $1.2$ and compare to Definition \ref{Definition 1.6}.
\end{theorem}

\begin{figure}[h]
\centerline{\psfig{figure=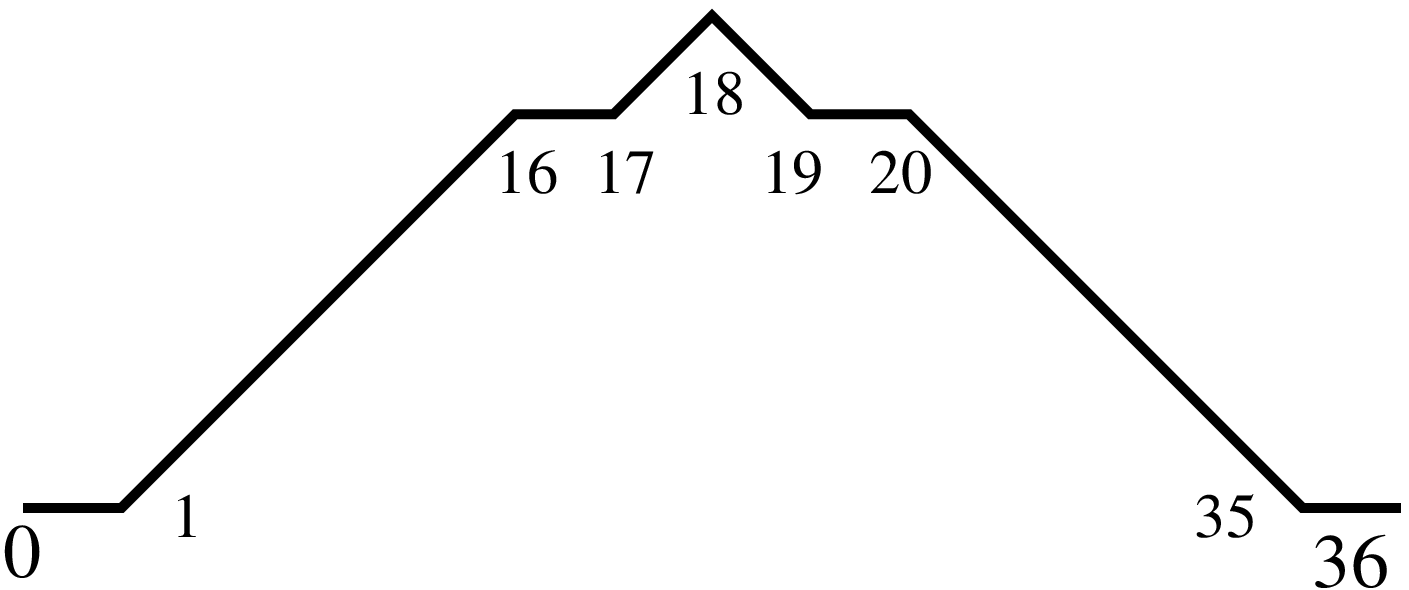,height=2.2cm}}\ \\ \
\centerline{Figure $1.2$: the shape of the polynomial ${12 \choose 6,6}_q$ called $(2,1,2)$ type on the top}
\end{figure}

Gaussian polynomials with $m\leq 4$ are carefully analyzed in Sections $2$ and $3$ (see also \cite{Lin,West}).

We complete this section with a fairly general result concerning the structure of a product of two unimodal polynomials.
Additionally, we define a preorder relation on unimodal polynomials and introduce the notion of their shapes.
\begin{proposition}\label{Proposition 1.5}\
\begin{enumerate}
\item[(1)] For the product of two $q$-integers, we have:
$$[m+1]_q[n+1]_q= $$
$$1+2q+3q^2+...+(m+1)q^m+...+(m+1)q^n+...+2q^{m+n-1}+q^{m+n}=$$
$$\sum_{i=0}^m q^i[m+n+1-2i]_q,$$
for $m\leq n$. We say that the product has a trapezoidal shape with base of length $m+n$ and top of length $n-m$
($n-m+1$ terms), see Figure $1.3$.
\item[(2)]
Let $Q(q)\in PSU_N$ thus it can be written as
$$Q(q)= \sum_{i=0}^Nc_iq^i =\sum_{i=0}^{\lfloor N/2 \rfloor}b_iq^{i}[N+1-2i]_q,$$
where $b_i=c_i-c_{i-1}\geq 0$  for $i>0$ and $b_0=c_0>0$.\\
Let $Q'(q)\in PSU_{N'}$ and
$$ Q'(q)=\sum_{j=0}^{N'}c'_jq^j =\sum_{j=0}^{\lfloor N'/2 \rfloor}b'_jq^{j}[N'+1-2j]_q.$$
Then the product $Q(q)Q'(q)$ is a polynomial in $PSU_{N+N'}$,
$$Q(q)Q'(q)= \sum_{k=0}^{\lfloor (N+N')/2 \rfloor} d_kq^{k}[N+N'+1-2k]_q.$$
Then $d_k \neq 0$ if and only if there are $i$ and $j$ with $b_i\neq 0$, $b'_j\neq 0$ and $k$ is in the interval
$[i+j, i+j + min(N-2i,N'-2j)]$.
%$[|i-j|,i+j]$.
\end{enumerate}
\end{proposition}
\begin{proof} (1) follows directly from the definition of multiplication.\\
(2) follows by using (1) several times.
\end{proof}

We stress that the simple observation in the proposition is  very important and used several times in this paper.

\begin{figure}[h]
\centerline{\psfig{figure=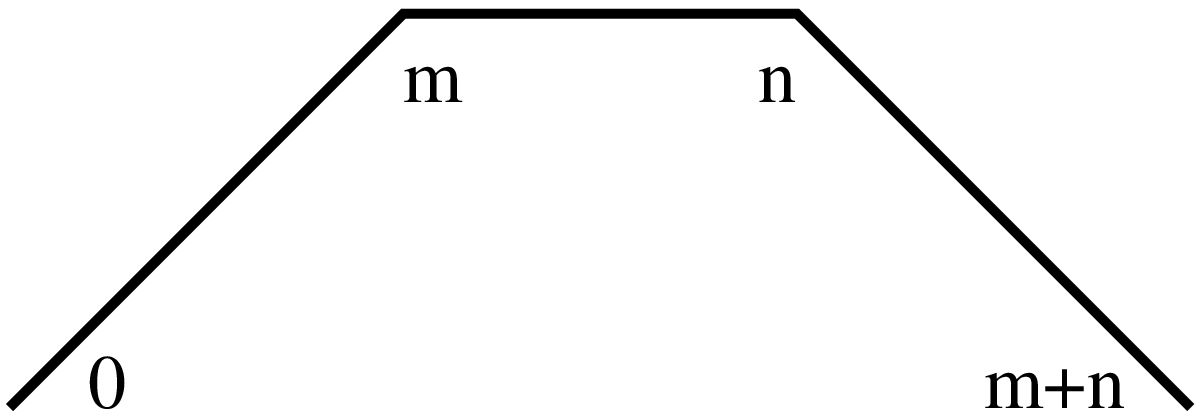,height=2.2cm}}\ \\ \ \\
\centerline{Figure $1.3$: \ a trapezoidal shape of $[m+1]_q[n+1]_q$.}
\end{figure}

Here we now define the notion of polynomials having the same shape (relation $\simeq$).\\
We consider the following preordering relation on nonnegative symmetric unimodal polynomials of variable $q$. We allow here
$b_0$ or $b'_0$ to be equal to zero.
\begin{definition}\label{Definition 1.6} Let
$$ P(q)= b_0[N+1]_q +b_1q[N-1]_q+b_2q^2[N-3]_q+...+b_{\lfloor N/2 \rfloor}q^{\lfloor N/2 \rfloor}[N+1 -2\lfloor N/2 \rfloor]_q $$
and
$$ P'(q)= b'_0[N+1]_q +b'_1q[N-1]_q+ b'_2q^2[N-3]_q+...+b'_{\lfloor N/2 \rfloor}q^{\lfloor N/2 \rfloor}[ N+1 - 2\lfloor N/2 \rfloor]_q $$
then we say that $P(q) \preceq P'(q)$ (we say the shape of $P'(q)$ dominates the shape of $P(q)$) if whenever $b'_i=0$ then $b_i=0$.
If $P(q) \preceq P'(q)$ and $P'(q) \preceq P(q),$ we say that the polynomials are shape equivalent or shortly they have the same shape,
and write $P(q) \simeq P'(q)$.\\
If $b_i\neq 0,$ then we say that $P(q)$ has a nontrivial row of length $N-2i$ at the height $i$.
\end{definition}

Proposition \ref{Proposition 1.5}, despite its simplicity, immediately leads to useful properties:
\begin{corollary}\label{Corollary 1.7}
\begin{enumerate}
\item[(1)]
Assume that $P(q)\in PSU_N$ and $[n+1]_q$ is a $q$-integer with $n \geq N$. Then the product
$P(q)[n+1]_q$ has a trapezoidal shape with a bottom of length $N+n$ and top of length $n-N$. Furthermore, if $n=N-1,$ then
$P(q)[n+1]_q$ is strictly unimodal. In other words, in these cases, $P(q)[n+1]_q$ has the same shape as $[N+1]_q[n+1]_q$.
\item[(2)] If $P_1(q), P_2(q) \in PSU_N$, $P_1(q) $ dominates $P_2(q)$, and $P_3(q)$ is a positive symmetric unimodal polynomial such that
$P_2(q)P_3(q)$ is strictly unimodal, then $P_1(q)P_3(q)$ is strictly unimodal.
\end{enumerate}
\end{corollary}

Another very simple corollary (or its variations) will be often used:
\begin{corollary}\label{Corollary 1.8}
Consider $P(q)=\sum\limits_{i=0}^{\lfloor N/2 \rfloor}b_iq^{i}[N+1-2i]_q\in PSU_N$ with $b_1= 0=b_s$ (for some $s\geq 3$) and otherwise $b_i\neq 0$.
Then for $k\leq N+1$, $P(q)[k+1]_q$ is strictly unimodal, except for $k+1=N-1$ or $k+1= N+1-2s$, where $P(q)[k+1]_q$ has a trapezoidal shape with a top of length $2$.
\end{corollary}

%PSEUDO Trapezoidal\\
We also obtain the following concrete corollary which will be one of the basic  bricks for our general results on strict unimodality of plucking polynomials.
\begin{corollary}\label{Corollary 1.9}
Let $P(q)$ be one of the exceptional polynomials in \cite{Pak-Pan} and $d=deg (P(q))$ then
$$P(q){m+n\choose m, n}_{q} \mbox{ where $1\leq m\leq n\geq 2$} $$ is strictly unimodal with the following exceptions:
\begin{enumerate}
\item[$(1)$] ${12 \choose 6,6}_q[3]_q$ which has a trapezoidal shape with a top of length $2$.
\item[$(2)$] $P(q) [d-1]_q$ which has a trapezoidal shape with a top of length $2$.
\item[$(3)$] $P(q)[k+1]_q$ ($k\geq d+2$) which has a trapezoidal shape with a top of length $k-d$.
\end{enumerate}
\end{corollary}
\begin{proof} The cases of $m=2,3,$ and $4$ are discussed in following sections. For $m \geq 5,$ the corollary follows from Theorem \ref{Theorem 1.4}
and Proposition \ref{Proposition 1.5}. The case of $m=1$ ($q$-integer $[n+1]_q$) follows from
Corollaries \ref{Corollary 1.7} and \ref{Corollary 1.8}. For $(1),$ we can also concretely compute that:
$${12 \choose 6,6}_q[3]_q= 1 + 2 q +...+ 161 q^{17} + 168 q^{18} + 168 q^{19} + 168 q^{20} + 161 q^{21} +... + q^{38}.$$
\end{proof}

The next two sections are purely algebraic. Our starting point is Theorem \ref{Theorem 1.2} decomposing $Q(T)$ into a product of
Gaussian $q$-binomial coefficients. Because of Theorem \ref{Theorem 1.4} of Pak and Panova, we have to pay special attention to factors
of the type ${m+n\choose m,n}_q$ with $m\leq 4$; especially because by Proposition \ref{Proposition 1.5} we can already conclude that
products of factors with $m \geq 5$ are strictly unimodal.

\section{Analysis of products of $q$-integers and Gaussian polynomials ${2+n\choose 2,n}_{q}$}

We start with an analysis of products of $q$-integers which will always have a trapezoidal shape.
Our results are based on Proposition \ref{Proposition 1.5}.

\begin{proposition}\label{Proposition 2.1}
Consider the sequence of positive numbers $a_1 \leq a_2 \leq ...\leq a_k$. Then the product
$[a_1+1]_q[a_2+1]_q \cdots [a_k+1]_q$ is always trapezoidal with base of length $a_1+a_2+...+a_k$ and
\begin{enumerate}
\item[(1)] strictly unimodal if and only if $a_1+a_2+...+a_{k-1} +1 \geq a_k$,
\item[(2)] trapezoidal with a top of length $a_k-(a_1+...+a_{k-1})$ if \\
 $a_k-(a_1+...+a_{k-1}) \geq 0$.
\end{enumerate}
\end{proposition}
\begin{proof} It holds for $k=2$ as $[a_1+1]_q[a_2+1]_q$ is a trapezoid with base of length $a_1+a_2$ and top of length $a_2-a_1$ by
Proposition \ref{Proposition 1.5}$(1).$
Thus for strict unimodality we need $a_2=a_1$ or $a_2=a_1+1$. To complete the proof inductively, we only need to prove the case
of three integers which can be restated as computing the product of a polynomial of trapezoidal shape with a $q$-integer. The inductive step reduces to
this case. Therefore, we only need part $(2)$ of the following lemma.
\end{proof}
\begin{lemma}\label{Lemma 2.2}  Assume $a\leq b\leq c$ and $b-a$ is even. Then:
\begin{enumerate}
\item [(1)] (Absorption law) The shape of the polynomial $(q^{(b-a)/2}[a+1]_q+ [b+1]_q)[c+1]_q$ is the same as that of $[b+1]_q[c+1]_q,$ that is a trapezoidal shape
 of base of length $b+c$ and top of length $c-b$.
\item[(2)] Let the polynomial $P_{a,b}(q)$ have a trapezoidal shape with the base of length $a+b$ and top of length $b-a$.
Then the product polynomial $P_{a,b}(q)[c+1]_q$ has a trapezoidal shape with base of length $a+b+c$ and the same shape as the triple
product $[a+1]_q[b+1]_q[c+1]_q$.
\item[(3)] Let the polynomials $P_{a,b}(q)$ ($a\leq b$) and $P_{c,d}(q)$ ($c \leq d$) have trapezoidal shapes with base of length $b+a$ and
$d+c$, respectively, and tops of length $b-a$ and $d-c$, respectively. Then the product $P_{a,b}(q)P_{c,d}(q)$ has a trapezoidal shape
of base of length $a+b+c+d$ and top of length which can be written as $2 \max(a,b,c,d)- (a+b+c+d)$ if this number is not negative.
Otherwise the product is strictly unimodal.
\end{enumerate}
\end{lemma}

\begin{proof}
$(1)$ follows directly from definition (it is a special case of Corollary \ref{Corollary 1.7}).\\
$(2)$ follows from $(1)$.\\
$(3)$ follows by applying $(2)$ twice and observing that $P_{c,d}(q)$ has the same shape as $[c+1]_q[d+1]_q$.
\end{proof}

We show here that the product of ${2+n\choose 2,n}_q$ and an integer $[k+1]_q=1+q+...+q^k$ ($k\geq 1$) is a polynomial with
a trapezoid shape. More precisely:
\begin{proposition}\label{Proposition 2.3}
${2+n \choose 2,n}_q[k+1]_q$ is strictly unimodal if $k+2\leq  2n$ and $2n+2-k$ is not divisible by $4$.
If $k+2\leq  2n$ and $2n+2-k$ is divisible by $4,$ then the product has a trapezoidal shape with a top of length $2$.
If $k \geq 2n,$ then the product has a trapezoidal shape with a top of length $k-2n$.
\end{proposition}
\begin{proof}
 We compare summands of
$${2+n \choose 2,n}_q= \sum_{i=0}^{\lfloor n/2 \rfloor}q^{2i}[2n+1-4i]_q  $$
with $[k+1]_q$. We see that if $2\leq k+1 \leq 2n+1,$ then the product has a trapezoidal shape with a top of length $0,1,$ or $2$, where
it is $2$ if and only if $2n+1 -(k+1) = 2n-k$ is congruent to $2$ modulo $4$. If $k \geq 2n,$ then ${2+n \choose 2,n}_q [k+1]_q$ has the same
shape as $[2n+1]_q[k+1]_q$ (Lemma \ref{Lemma 2.2} (1)) so it is of the shape of trapezoid with a top of length $k-2n$, as needed.
\end{proof}
As a small illustration of the proposition we compute that
$${12 \choose 2,10}_q [3]_q= $$
$$1 + 2 q + 4 q^2 + 5 q^3 + 7 q^4 + 8 q^5 + 10 q^6 + 11 q^7 + 13 q^8 + 14 q^9 +$$
$$ 16 q^{10} + 16 q^{11} + 16 q^{12} + ...+q^{22}$$
has a trapezoidal shape with a top of length $2$ and is therefore not strictly unimodal.

We now show that if we multiply two Gaussian polynomials of type ${2+n\choose 2,n}_q,$ then we get either a strictly unimodal polynomial or
a polynomial of a trapezoidal shape with a top of length $2$.
\begin{proposition}\label{Proposition 2.4} The product ${2+n \choose 2,n}_q{2+m \choose 2,m}_q$ is
strictly unimodal if $m$ and $n$ have the same parity and is of trapezoidal shape with a top of length $2$ ($3$ terms) if they have different parity.
\end{proposition}
\begin{proof} We use the fact that Gaussian polynomials ${2+n \choose 2,n}_q$ have a shape of a step pyramid.
Separating even and odd cases, we have:
$${2+n \choose 2,n}_q = \sum_{i=0}^{\lfloor n/2 \rfloor}q^{2i}[2n+1-4i]_q = $$
$$ \left \{ \begin{array}{ll}
   [2n+1]_q +q^2[2n-3]_q +...+q^{n}[1]_q  & \mbox{if $n$ is even;} \\
  \ [2n+1]_q +q^2[2n-3]_q +...+q^{n-1}[3]_q  & \mbox{if $n$ is odd.}
\end{array}
              \right. $$
We complete our proof by using Proposition \ref{Proposition 1.5} $(2).$
\end{proof}
We illustrate Proposition \ref{Proposition 2.4} by two examples:
$$ {6 \choose 2,4}_q{4 \choose 2,2}_q=$$
 $$ 1 + 2 q + 5 q^2 + 7 q^3 + 11 q^4 + 12 q^5 + 14 q^6 + 12 q^7 +
 11 q^8 + 7 q^9 + 5 q^{10} + 2 q^{11} + q^{12}.$$
(It is strictly unimodal polynomial).
$$ {4 \choose 2,2}_q{5 \choose 2,3}_q=$$
 $$1 + 2 q + 5 q^2 + 7 q^3 + 10 q^4 + 10 q^5 + 10 q^6 + 7 q^7 + 5 q^8 + 2 q^9 + q^{10}.$$
(The polynomial has a trapezoidal  shape with a top of length $2$.)

\begin{corollary}\label{Corollary 2.5}
The product of three terms  ${2+n \choose 2,n}_q{2+m \choose 2,m}_q{2+k\choose 2,k}_q$ is always strictly unimodal.
\end{corollary}
\begin{proof} Two terms of the product have the same parity, so their product is strictly unimodal. Then we deduce from Proposition \ref{Proposition 1.5} $(2)$
 that the product of a strictly unimodal (PSU) polynomial and any $q$-binomial coefficient ${2+n \choose 2,n}_q$ is strictly unimodal.
\end{proof}

We can also observe that by Proposition \ref{Proposition 1.5} (2), the product of ${2+n \choose 2,n}_q$ ($n\geq 2$) and any Gaussian polynomial
with $5\leq m\leq n$ is strictly unimodal. In the next section we analyze the case of $m=3$ or $4$.

\section{Analysis of  ${3+n\choose 3,n}_{q}$ and ${4+n\choose 4,n}_{q}$}\label{Section 3}
We consider the set $L(m,n)$ which consists of integer sequences of length $m$ denoted by ${\bf a}=(a_{1},\ldots,a_{m})$
such that $0 \leq a_{1} \leq \cdots \leq a_{m} \leq n$
with ordering $(a_{1},\ldots,a_{m}) \leq (b_{1},\ldots,b_{m})$ if $a_{i} \leq b_{i}$ for every $i.$ A chain ${\bf a}_{1}< \cdots < {\bf a}_{k}$
is called \emph{symmetric} if $r({\bf a}_{1})+r({\bf a}_{k})=mn$ where $r({\bf a})=\sum\limits_{i=1}^{m}a_{i}.$
We define a subset $S(m,n)$ of $L(m,n)$ by $S(m,n)=\{(a_{1},\ldots,a_{m}) \in L(m,n) | a_{1}=0 \hbox{~or~} a_{m}=n \}.$
Then the complement of $S(m,n)$ in $L(m,n)$ is isomorphic to $L(m,n-2).$

Notice that each $L(m,n)$ corresponds to a $q$-binomial coefficient ${m+n \choose m,n}_{q}$:
$${m+n \choose m,n}_{q}=\sum\limits_{{\bf a} \in L(m,n)}q^{r({\bf a})}.$$
For an elementary approach to the relation between $L(m,n)$ and $q$-binomial coefficients ${m+n \choose m,n}_{q}$ we refer to \cite{Sta-2}.

Lindstr\"om\cite{Lin} and West\cite{West} introduced symmetric chain decompositions of $S(3,n)$ and $S(4,n),$ respectively.
We will classify the shapes of the $q$-polynomials ${3+n \choose 3,n}_{q}$ and ${4+n \choose 4,n}_{q}$ by using their symmetric
chain decompositions of $L(3,n)$ and $L(4,n),$ respectively.

\begin{lemma}\label{Lemma 3.1}
For $n\geq0,$ $q$-polynomial ${3+n \choose 3,n}_{q}=\sum\limits_{i=0}^{3n} c_{i}q^{i}$ has one of the following forms:
\begin{enumerate}
  \item if $n =2k+1,$ $$c_{0}=c_{1}<\cdots<c_{3k}=c_{3k+1}=c_{3k+2}=c_{3k+3}>\cdots>c_{6k+2}=c_{6k+3},$$
that is the polynomial has an almost trapezoidal shape with a top of length $3$ ($4$ terms),
  \item if $n =4k,$ $$c_{0}=c_{1}<\cdots<c_{6k-2}=c_{6k-1}<c_{6k}>c_{6k+1}=c_{6k+2}>\cdots>c_{12k-1}=c_{12k},$$
and we say that the top of the polynomial has a shape of type $(2,1,2)$, the same shape as ${12\choose 6,6}_q$ (Figure $1.2$),
  \item if $n =4k+2,$ $$c_{0}=c_{1}<\cdots<c_{6k}=c_{6k+1}<c_{6k+2}=c_{6k+3}=c_{6k+4}>c_{6k+5}=c_{6k+6}>\cdots$$ $$>c_{12k+5}=c_{12k+6},$$
and we say that the top of the polynomial has a shape of type $(2,3,2)$, compare Figure $3.1$.
\end{enumerate}
\end{lemma}
To illustrate our lemma we see that:
$${6 \choose 3, 3}_q = 1+q +2q^2 + 3q^3 + 3q^4 +3q^5 + 3q^6 + 2q^7 + q^8 +q^9.$$
$${7 \choose 3, 4}_q =  1+q +2q^2 + 3q^3 + 4q^4 +4q^5 + 5q^6 + 4q^7 + 4q^8 +3q^9 + 2q^{10} + q^{11} +q^{12}. $$
$${9  \choose 3,6}_q =  1+ q  + 2 q^2 + 3q^3 + 4 q^4 + 5 q^5 + $$
$$7 q^6 + 7 q^7 + 8 q^8 + 8 q^9 + 8 q^{10} + 7 q^{11} + 7 q^{12} +\ldots + q^{18};$$
the shape with a top of type $(2,3,2)$ as illustrated in Figure $3.1$. \ \\ \ \\

\begin{figure}[h]
\centerline{\psfig{figure=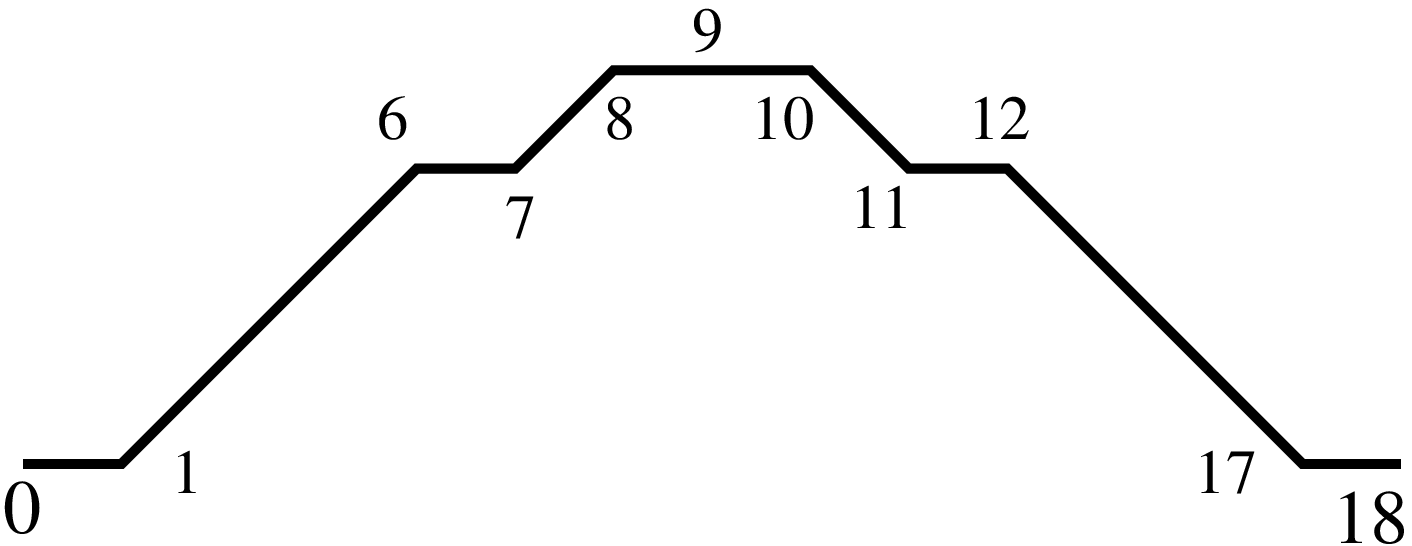,height=2.4cm}}\ \\ \
\centerline{Figure $3.1$: shape of the polynomial ${4k+1\choose 3, 4k-2}_q $; e.g.,  for $k=2$: ${9  \choose 3,6}_q$}
\end{figure}

\begin{proof}
We separate the proof into three cases.\\
Case I. $n =2k+1.$ Lindstr\"om\cite{Lin} showed that $S(3,n)$ can be decomposed into
symmetric chains $C_{i}^{n},$ for $i=0,1,\cdots,k,$ with length $3n-4i+1.$

Since $L(3,n) \setminus S(3,n)$ is isomorphic to $L(3,n-2),$ we have
$$L(3,n)\cong \bigcup\limits_{j=1}^{\frac{n-1}{2}} S(3,2j+1) \cup L(3,1)\cong\bigcup\limits_{j=0}^{\frac{n-1}{2}} S(3,2j+1).$$
We can determine coefficients $c_{i}$'s of $q$-polynomial ${3+n \choose 3,n}_{q}$ by checking lengths of chains $C_{i}^{2j+1}$ because they are
symmetric chains. Notice that the coefficient $c_{i}$ represents the number of integer sequences ${\bf a}$ in $L(3,n)$ satisfying
$r({\bf a})=i.$ Since $|C_{i}^{2j+1}|=6j-4i+4$ and $0 \leq i \leq j \leq k,$ the set of all lengths of $C_{i}^{2j+1}$ in $L(3,n)$
is the same with the set of all positive even integers which are less than or equal to $6k+4$ but neither $2$ nor $6k+2.$ Therefore, we have
$$c_{0}=c_{1}<\cdots<c_{3k}=c_{3k+1}=c_{3k+2}=c_{3k+3}>\cdots>c_{6k+2}=c_{6k+3}.$$
Before we consider the other two cases, we define another subset $T(3,n)$ of $L(3,n)$ by
$T(3,n)=\{(a_{1},a_{2},a_{3}) \in L(3,n) | a_{1}=1 \hbox{~or~} a_{3}=n-1 \}.$
Then the complement of $S(3,n) \cup T(3,n)$ in $L(3,n)$ is isomorphic to $L(3,n-4).$ Thus we have
$$L(3,n)\cong\left\{\begin{array}{ll}
                    \bigcup\limits_{j=0}^{\frac{n}{4}} \{S(3,4j) \cup T(3,4j)\}, & \hbox{if $n=4k$;} \\
                    \bigcup\limits_{j=0}^{\frac{n-2}{4}} \{S(3,4j+2) \cup  T(3,4j+2)\}, & \hbox{if $n=4k+2$.}
                  \end{array}
                \right.$$

Lindstr\"om\cite{Lin} proved that $S(3,n) \cup T(3,n)$ can be decomposed into three types of symmetric chains $C^{n},$ $D_{i}^{n},$ and $E_{i}^{n}$
with lengths $n+1,$ $3n-4i+1,$ and $3n-4i-1,$ respectively:
$$S(3,n) \cup T(3,n)=\{C^{n}\} \cup \{D_{i}^{n} | 0 \leq i \leq \frac{n}{2}-1 \} \cup \{E_{i}^{n} | 1 \leq i \leq \frac{n}{2}-1 \}.$$

Case II. $n=4k.$ Notice that $|C^{4j}|=4j+1,$ $|D_{i}^{4j}|=12j-4i+1,$ and $|E_{i}^{4j}|=12j-4i-1.$ By the same argument with Case I,
if we check the set of all lengths of the symmetric chains $C^{4j},$ $D_{i}^{4j},$ and $E_{i}^{4j}$ in $L(3,n),$ then we can see that it is
equal to $\{2l+1 \in \mathbb{Z} | 0 \leq l \leq 6k \} \backslash \{ 3, 12k-1 \}.$ Thus we have
$$c_{0}=c_{1}<\cdots<c_{6k-2}=c_{6k-1}<c_{6k}>c_{6k+1}=c_{6k+2}>\cdots>c_{12k-1}=c_{12k}.$$
Case III. $n=4k+2.$ Note that $|C^{4j+2}|=4j+3,$ $|D_{i}^{4j+2}|=12j-4i+7,$ and $|E_{i}^{4j+2}|=12j-4i+5.$ Again, since the set of all
lengths of the symmetric chains $C^{4j+2},$ $D_{i}^{4j+2},$ and $E_{i}^{4j+2}$ in $L(3,n)$ is the same
with $\{2l+1 \in \mathbb{Z} | 0 \leq l \leq 6k+3 \} \backslash \{ 1, 5, 12k+5 \},$ we have
$$c_{0}=c_{1}<\cdots<c_{6k}=c_{6k+1}<c_{6k+2}=c_{6k+3}=c_{6k+4}>c_{6k+5}=c_{6k+6}>\cdots$$ $$>c_{12k+5}=c_{12k+6}.$$
\end{proof}

Lemma \ref{Lemma 3.1} combined with Proposition \ref{Proposition 1.5} gives the following useful corollary which we use in Section 4.
\begin{corollary}\label{Corollary 3.2}
The product
$${3+n \choose 3,n}_q{m_2+n_2\choose m_2,n_2}_q \mbox{ for $n\geq 3,\ m_2,n_2 \geq 1$},$$
is strictly unimodal with the following exceptions:
\begin{enumerate}
\item[(1)] $${4k+1 \choose 3,4k-2}_q{4 \choose 2,2}_q \mbox{ for $k\geq 2$ see Figure $3.2$}$$
The product polynomial has a trapezoidal shape with a top of length $2$ ($3$ terms).
\item[(2)] $${4k+1 \choose 3,4k-2}_q[5]_q \mbox{ for $k\geq 2$}.$$
The product polynomial has a trapezoidal shape with a top of length $2$ ($3$ terms).
\item[(3)] $${3+4k \choose 3,4k}_q[3]_q \mbox{ for $k\geq 1$}.$$
The product polynomial has a trapezoidal shape with a top of length $2$ ($3$ terms).
\item[(4)] $${2k \choose 3,2k-3}_q[2]_q \mbox{ for $k\geq 3$}.$$
The product polynomial has a trapezoidal shape with a top of length $2$ ($3$ terms).
\item[(5)] $${3+n \choose 3,n}_q[3n-1]_q \mbox{ for $n\geq 3$}.$$
The product polynomial has a trapezoidal shape with a top of length $2$ ($3$ terms).
\item[(6)] $${3+n \choose 3,n}_q[3n+1+k]_q \mbox{ for $n\geq 3, k\geq 2$}.$$
The product polynomial has a trapezoidal shape with a top of length $k$ ($k+1$ terms).
\end{enumerate}
 \end{corollary}

\begin{example}\label{Example 3.3}
$${9 \choose 3,6}_q{4 \choose 2,2}_q= $$
$$1 + 2 q + 5 q^2 + 8 q^3 + 13 q^4 + 18 q^5 + 25 q^6 + 31 q^7 + 38 q^8 + 42 q^9 + $$
$$46 q^{10} + 46 q^{11} + 46 q^{12} +...+q^{22}.$$
(Trapezoidal shape with a top of length $2$ ($3$ terms)).
\end{example}

\begin{figure}[h]
\centerline{\psfig{figure=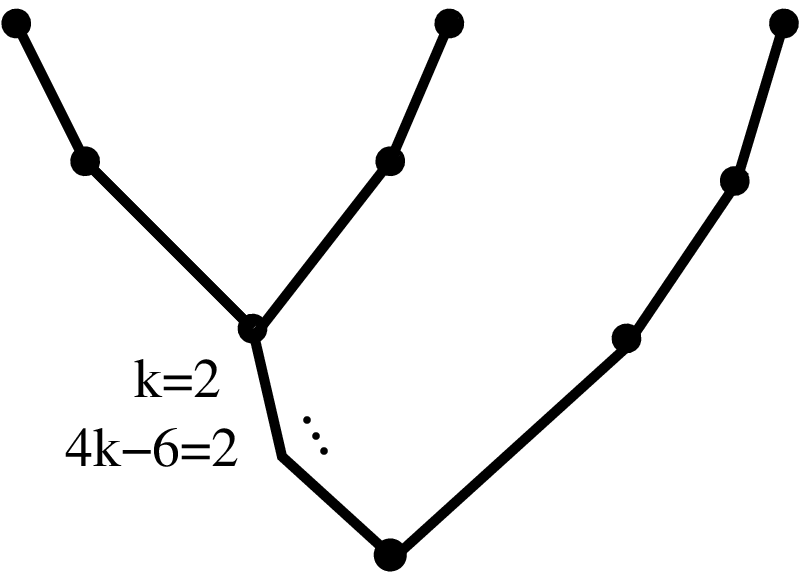,height=2.4cm}}\ \\ \ \\
\centerline{Figure $3.2$: \ $Q(T)= {4k+1 \choose 3,4k-2}_q{4 \choose 2,2}_q$; the case of $4k+1=9$}
\end{figure}

We now concentrate on the Gaussian polynomials ${4+n \choose 4,n}_{q}$.

\begin{lemma}\label{Lemma 3.4}
For $n \in \mathbb{N}\setminus\{1,4\},$ $q$-polynomial ${4+n \choose 4,n}_{q}=\sum\limits_{i=0}^{4n} c_{i}q^{i}$ has the form:
$$c_{0}=c_{1}<\cdots<c_{2n-2}=c_{2n-1}<c_{2n}>c_{2n+1}=c_{2n+2}>\cdots>c_{4n-1}=c_{4n},$$
that is the polynomial with the top of the shape having type $(2,1,2)$,
\end{lemma}

\begin{proof}
Note that ${5 \choose 4,1}_{q}=1+q+q^{2}+q^{3}+q^{4}= [5]_q$ and
$${8 \choose 4,4}_{q}=$$
$$1+q+2q^{2}+3q^{3}+5q^{4}+5q^{5}+7q^{6}+7q^{7}+8q^{8}+7q^{9}+7q^{10}+5q^{11}+5q^{12}+3q^{13}+2q^{14}+q^{15}+q^{16}.$$

West\cite{West} proved that $S(4,n)$ can be decomposed with two types of symmetric chains $C_{ij}^{n}$ and $D_{ij}^{n}$ of lengths $4(n-3i-j)+1$
and $4(n-3i-j)-5,$ respectively:
$$S(4,n)=\{C_{ij}^{n} | 3i+2j \leq n,~ i,j \geq 0 \} \cup \{D_{ij}^{n} | 3i+2j \leq n-3,~ i,j \geq 0 \}.$$
Since chains $C_{ij}^{n}$ and $D_{ij}^{n}$ are symmetric, we only need to consider their lengths to determine coefficients $c_{i}$'s of $q$-polynomial
${4+n \choose 4,n}_{q}$. Here the coefficient $c_{i}$ represents the number of integer sequences ${\bf a}$ in $L(4,n)$ satisfying $r({\bf a})=i.$

Note that $|C_{ij}^{n}| \equiv 1$ (mod $4$) and $|D_{ij}^{n}| \equiv 3$ (mod $4$).

First, we easily check that the length of any chain cannot be $3$ or $4n-1$ because $|D_{ij}^{n}|=4(n-3i-j)-5$ and $3i+2j \leq n-3.$

Since $L(4,n-2)\cong L(4,n)\setminus S(4,n),$ we have
$$L(4,n)\cong\left\{
                  \begin{array}{ll}
                    \bigcup\limits_{k=1}^{\frac{n}{2}} S(4,2k) \cup L(4,0)\cong\bigcup\limits_{k=0}^{\frac{n}{2}} S(4,2k), & \hbox{if $n$ is even;} \\
                    \bigcup\limits_{k=1}^{\frac{n-1}{2}} S(4,2k+1) \cup L(4,1)\cong\bigcup\limits_{k=0}^{\frac{n-1}{2}} S(4,2k+1), & \hbox{if $n$ is odd.}
                  \end{array}
                \right.$$

Then the set of all lengths of $C_{ij}^{n}$ and $D_{ij}^{n}$ in $L(4,n)$ is the same as the set of all positive odd integers which are less than or equal
to $4n+1$ but neither $3$ nor $4n-1$ if $n \in \mathbb{N}\setminus\{1,4\}.$ Therefore
$$c_{0}=c_{1}<\cdots<c_{2n-2}=c_{2n-1}<c_{2n}>c_{2n+1}=c_{2n+2}>\cdots>c_{4n-1}=c_{4n}.$$
\end{proof}

Lemma \ref{Lemma 3.4} combined with Proposition \ref{Proposition 1.5} gives the following useful corollary which we use in Section 4.
\begin{corollary}\label{Corollary 3.5}
The product polynomial
$${4+n \choose 4,n}_q{m_2+n_2\choose m_2,n_2}_q \mbox{ for $n\geq 4,\ m_2,n_2 \geq 1$},$$
is strictly unimodal with the following exceptions:
\begin{enumerate}
\item[(1)] $${8 \choose 4,4}_q{5 \choose 2,3}_q.$$
The product polynomial has a trapezoidal shape with a top of length $2$ ($3$ terms).
\item[(2)] $${8 \choose 4,4}_q[7]_q.$$
The product polynomial has a trapezoidal shape with a top of length $2$ ($3$ terms).
\item[(3)]$${4+n \choose 4,n}_q[3]_q \mbox{ for $n\geq 4$},$$
The product polynomial has a trapezoidal shape with a top of length $2$ ($3$ terms).
\item[(4)] $${4+n \choose 3,n}_q[4n-1]_q \mbox{ for $n\geq 4$}.$$
The product polynomial has a trapezoidal shape with a top of length $2$ ($3$ terms).
\item[(5)] $${4+n \choose 3,n}_q[4n+1+k]_q \mbox{ for $n\geq 4, k\geq 2$}.$$
The product polynomial has a trapezoidal shape with a top of length $k$ ($k+1$ terms).
\end{enumerate}
 \end{corollary}

\begin{proof}
Lemma \ref{Lemma 3.4} follows directly from Proposition \ref{Proposition 1.5}
but we will do well a little longer on the case of ${8 \choose 4,4}_q$. We can express it as
$${8 \choose 4,4}_{q}=$$
$$[17]_q + q^2[13]_q + q^3[11]_q+ 2q^4[9]_q +2q^{6}[5]_q + q^8[1]_q.$$
We can draw the shape of the polynomial as in Figure $3.3.$

\begin{figure}[h]
\centerline{\psfig{figure=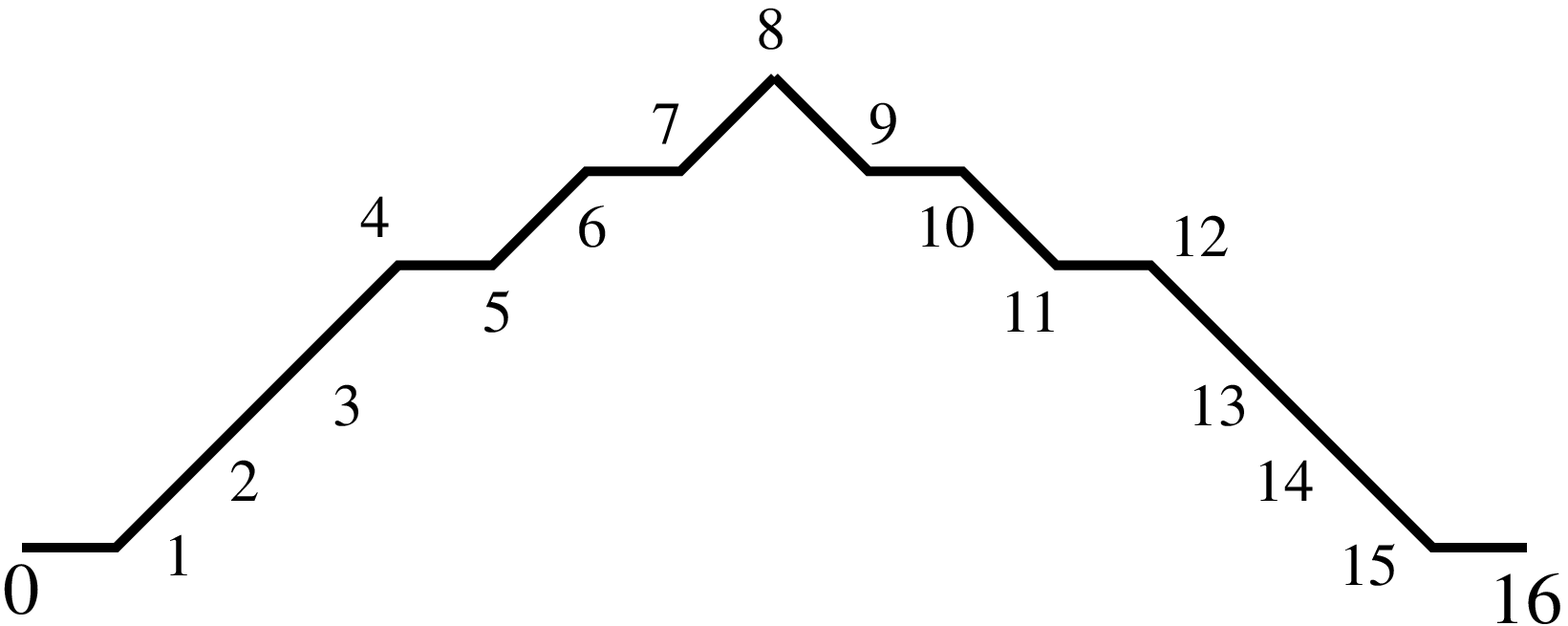,height=3.6cm}}\ \\ \ \\
\centerline{Figure $3.3$: the shape of ${8 \choose 4,4}_{q}$, the rows of high $1$, $5$ and $7$ are missing}
\end{figure}

To have strict unimodality we had to have rows corresponding to $q[15]_q$, $q^5[7]_q$ and $q^7[3]_q$.
Thus clearly, by Proposition \ref{Proposition 1.5} ${8 \choose 4,4}_q[15]_q$, ${8 \choose 4,4}_q[7]_q$ and ${8 \choose 4,4}_q[3]_q$ have
a trapezoidal shape with a top of length $2$.
More interestingly, as $[7]_q+ q^2[3]_q= {5 \choose 2,3}_q$ so, by Proposition \ref{Proposition 1.5}, ${8 \choose 4,4}_q {5 \choose 2,3}_q$ is
also of a trapezoidal shape of top of length $2$. Other cases are treated similarly.
\end{proof}

\section{The main algebraic result}\label{Section 4}

We are ready to combine the results from previous sections to decide which products of Gaussian polynomials are strictly unimodal  and show that
all nontrivial products (more than one factor different from $1$) are of trapezoidal shape (if they are not strictly unimodal, then we
give the length of top of the trapezoid). Our algebraic results are summarized in the theorem below.

\begin{theorem}\label{Theorem 4.1}\
 Consider a nontrivial product (at least two factors different from $1$) of $q$-binomial coefficients:
$$P(q) = P_1(q)P_2(q)\cdots P_k(q), \mbox{ $k\geq 2$},$$
$$ \mbox{where } P_i(q) = {m_i+n_i \choose m_i,n_i}_q, \ 1\leq m_i \leq n_i \geq 2.$$
Then:
\begin{enumerate}
\item[(1)] The  product $P(q)$ is always of a trapezoidal shape.
\item[(2)] If the product has no $q$-integer factors, then it is always strictly unimodal except ${8 \choose 4,4}_q{5 \choose 2,3}_q$,
${4k+1 \choose 3,4k-2}_q{4 \choose 2,2}_q$, or if we multiply
two factors of type ${2+n \choose 2,n}_q$ with different parity of $n$ (Proposition \ref{Proposition 2.4}).
In these cases the top of the resulting trapezoid has length $2$.
\item[(3)] If the product has exactly one $q$-integer factor $[n+1]_q$, the product $P(q)$ is strictly unimodal with exceptions
of the cases of $2n \geq deg(P(q)) +2$ which have a trapezoidal shape with a top of length $2n- deg(P(q))$ and products
of the form $[n+1]_q{m_2+n_2\choose m_2,n_2}_q$, $2\leq m_2$ listed in Proposition \ref{Proposition 2.3}, Corollary \ref{Corollary 3.2}, and
Corollary \ref{Corollary 3.5}.
\item[(4)] The case of all factors being $q$-integers is described in Proposition  \ref{Proposition 2.1}. (See also Corollary \ref{Corollary 5.5}.)
\item[(5)]  All other cases when $P(q)$ is not strictly unimodal can be characterized as follows:\\
Let $[n+1]_q$ be the largest $q$-integer factor of $P(q)$ and let $P(q)=[n+1]_q\hat P(q)$. Assume also that
$n \geq deg(\hat P(q))+ 2$. Then $P(q)$ has a trapezoidal shape with a top of length $n- deg (\hat P(q))$.
\end{enumerate}
\end{theorem}
%${2+n_1\choose 2, n_1}_q {n_2\choose 2, n_2}_q$ ($n_1-n_2$ odd) has a trapezoidal shape with the top of length $2$ (Proposition \ref{Proposition 2.4}).

\begin{proof}
We already proved all of the main ingredients needed to demonstrate Theorem \ref{Theorem 4.1}.
\end{proof}

\section{Tree realization and future plans}\label{Section 5}

Our original problem was to characterize those rooted trees whose plucking polynomials are not strictly unimodal. We are not interested
in Gaussian polynomials ${b+a \choose b,a}_q$ (analyzed carefully in \cite{Pak-Pan} and follow up papers \cite{Dha,Zan}) which are plucking polynomials
of trees with one splitting (Figure $5.1$). Such a tree after reduction\footnote{We say that a tree is \emph{reduced} if it has no string, that is its root has degree different from $1.$ The reduction of the given tree is the reduced tree $T$ obtained from $T$ by cutting its string.} is denoted by $T_{b,a}$.
From our main algebraic result, Theorem \ref{Theorem 4.1}, we obtain:

\begin{corollary}
Every tree which is different from $T_{b,a}$ with a string (Figure $5.1$) has a plucking polynomial $Q(T)$ of a trapezoidal shape.
\end{corollary}

\begin{figure}[h]
\centerline{\psfig{figure=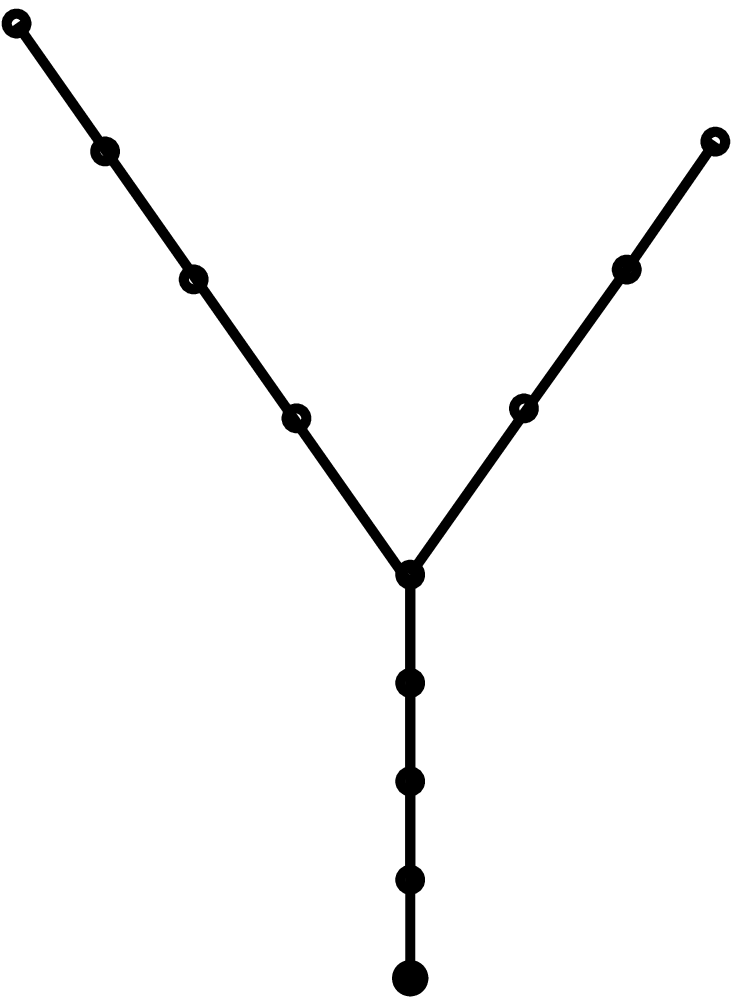,height=3cm}}
\centerline{Figure $5.1$: \ A tree of type $T_{b,a}$ with a string}
\end{figure}

Consider a tree $T$ whose reduction is not equal to $T_{b,a}.$ If we would like to decide whether $Q(T)$ is strictly unimodal and if it is not what the length of the top of the corresponding trapezoidal shape is, we can decompose $Q(T)$ into the product of Gaussian polynomials (not necessarily unique), as in Theorem \ref{Theorem 1.2}, and then use Theorem \ref{Theorem 4.1}.

We can however ask further questions:
\begin{enumerate}
\item[(1)] Which products of Gaussian polynomials can be realized by trees as $Q(T)$?,
\item[(2)] To what extent is the realization unique?
\end{enumerate}

We answer the first question in Theorem \ref{Theorem 5.6}. We discuss the second question in Subsection $5.2.$

\subsection{Realizations}

We start from the rather pleasing criterion for the product of Gaussian polynomials to be realized as $Q(T)$ for some $T.$ This will show, in particular, that ${8 \choose 4,4}_q{5 \choose 2,3}_q$ of Theorem \ref{Theorem 4.1} cannot be realized as the plucking polynomial of any tree.
We consider polynomials $P(q)$ which can be represented by a fraction for which the numerator and denominator are products of $q$-integers. Every plucking polynomial can be written in this form. We denote the numerator and denominator of $P(q)$ by $N(P(q))$ and $D(P(q))$, respectively.

\begin{definition}
A $q$-polynomial $P(q)=\frac{[a_{1}]_{q}\cdots [a_{k}]_{q}}{[b_{1}]_{q}\cdots [b_{l}]_{q}}$ is in a reduced form if $a_{i} \neq b_{j}$ for any $i,j.$
\end{definition}

It is well known that the reduced form is unique.

\begin{theorem}\label{Theorem 5.3}
Assume that the plucking polynomial $Q(T)$ of a rooted tree $T$ is in a reduced form. Then
\begin{enumerate}
  \item $N(Q(T))$ divides $[|E(T)|]_{q}!$ in $q$-integer sense (i.e. $N(Q(T))$ is composed of $q$-integer factors of $[|E(T)|]_{q}!$).
In particular, $N(Q(T))$ is square free,
  \item if $T=T_{1} \vee \cdots \vee T_{k}$ and $k \geq 2,$ then $N(Q(T))$ contains $q$-integer factors $[d]_{q}$ where $|E(T)| \geq d > \max\limits_{1 \leq i\leq k}(E_{i}).$ In particular, $|E(T)|$ is the greatest $q$-integer in $N(Q(T)).$
\end{enumerate}
\end{theorem}

\begin{proof}
$(1)$ Let $T$ be a rooted tree. We consider every vertex of degree $\geq 3,$ including the root if the degree of the root is $\geq 2,$ denoted by $v_{0},v_{1}, \ldots, v_{m}$ where $v_{0}$ is the root of the greatest subtree of $T$. Then
$$Q(T)=\prod_{v\in \{v_{0},v_{1},\ldots,v_{m}\}}\binom{E(T^v)}{E(T^v_{k_v}),...,E(T^v_{1})}_q$$
where $T^{v}$ is the subtree of $T$ with root $v$ and $k_{v_{i}}+1$ is the degree of $v_{i},$ especially if $v_{0}$ is the root of $T,$ then $k_{v_{0}}$ will be the degree of $v_{0}$ (compare with Theorem \ref{Theorem 1.2}$(3)$ ``State product formula"). Then in the reduced form $\widehat{Q}(T)$ of $Q(T),$ every $[E(T^{v_{i}})]_{q}!$ which is the numerator of $\binom{E(T^v)}{E(T^v_{k_v}),...,E(T^v_{1})}_q$ for $1 \leq i \leq m$ will be canceled out (because each subtree $T^{v_{i}}$ belongs to the next greater subtree $T^{v_{j}}$ therefore $E(T^{v_{i}}) \leq E(T^{v_{j}}_{t})$ for some $t \in \{1,\ldots, k_{v_{j}}\}$, i.e. $[E(T^{v_{i}})]_{q}!$ divides $[E(T^{v_{j}}_{t})]_{q}!$ in $q$-integer sense). Thus the numerator $N(Q(T))$ of $\widehat{Q}(T)$ has only some $q$-integer factors of $[E(T^{v_{0}})]_{q}!$ Clearly, $E(T^{v_{0}}) \leq |E(T)|$ therefore $N(Q(T))$ is composed of $q$-integer factors of $[|E(T)|]_{q}!$

$(2)$ Since $T=T_{1} \vee \cdots \vee T_{k}$ and $k \geq 2,$ the tree $T$ is reduced. Then by the similar argument in the proof of $(1),$ $q$-integer factors $[d]_{q}$ cannot be canceled if $|E(T)| \geq d > \max\limits_{1 \leq i\leq k}(E_{i}).$ Moreover, since $E(T^{v_{0}})=|E(T)|,$ the greatest $q$-integer in the numerator of the reduced form of $Q(T)$ is $|E(T)|.$
\end{proof}

\begin{example}
${8 \choose 4,4}_{q}{5 \choose 2,3}_{q}$ and ${5 \choose 2,3}_{q}{4 \choose 2,2}_{q}$ cannot be realized as a $q$-polynomial of some
rooted trees because ${8 \choose 4,4}_{q}{5 \choose 2,3}_{q}=\frac{[8]_{q}[7]_{q}[6]_{q}[5]_{q}^{2}}{[3]_{q}[2]_{q}^{2}}$
and ${5 \choose 2,3}_{q}{4 \choose 2,2}_{q}=\frac{[5]_{q}[4]_{q}^{2}[3]_{q}}{[2]_{q}^{2}}$ are in reduced forms.
Notice (Theorem \ref{Theorem 4.1}) that both $q$-polynomials are not strictly unimodal.
However, ${4k+1 \choose 3,4k-2}_q{4 \choose 2,2}_q$ for $k \geq 2$ can be realized as $Q(T)$ for some $T.$ (See Figure $3.2.$)
\end{example}

\begin{corollary}\label{Corollary 5.5} A product of $q$-integers $[a_1]_q[a_2]_q\cdots[a_k]_q$ where $a_1 \leq a_2 \leq ...\leq a_k$ can be realized as a $q$-polynomial of some rooted tree $Q(T)$ if and only if all inequality are strict, that is $a_1 < a_2 < ... < a_k.$
\end{corollary}

\begin{proof}
(if) We proceed by induction on $k,$ the number of factors in the product. $[a_1]_q$ can be realized by a tree.
Then in the inductive step we attach at the bottom the tree of size $a_k- a_{k-1}$ as illustrated in the Figure $5.2$
for $[4]_q[5]_q[8]_q[10]_q [11]_q$.\\
(only if) This directly follows from Theorem \ref{Theorem 5.3}$(1)$.
\end{proof}

\begin{figure}[h]
\centerline{\psfig{figure=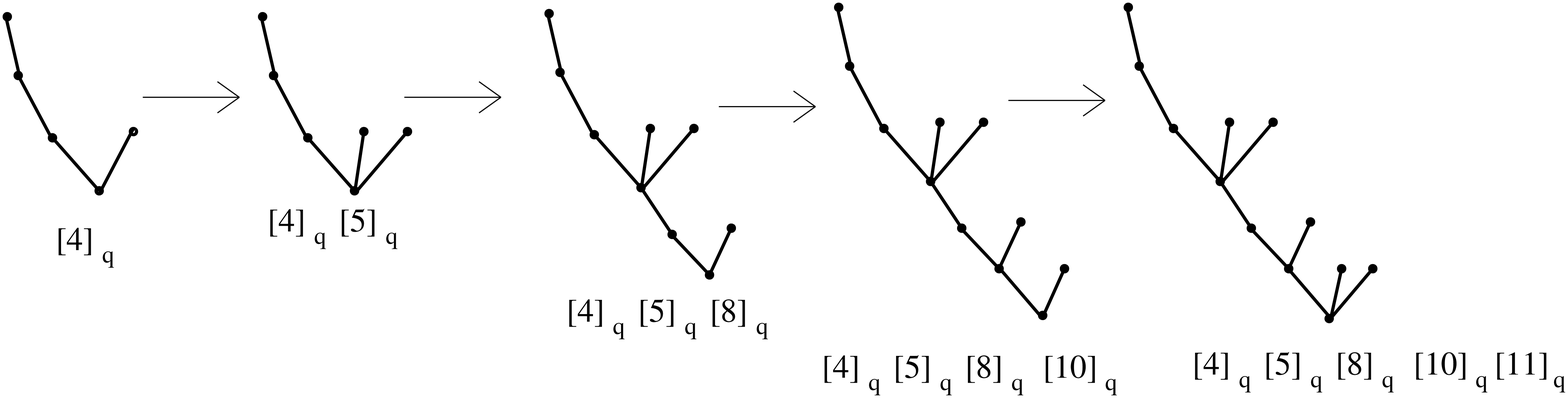,height=3.6cm}}\ \\
\centerline{Figure $5.2$: Realization of the polynomial $[4]_q[5]_q[8]_q[10]_q [11]_q$}
\end{figure}

As a generalization of Corollary 5.5, now we give a complete solution to the first characterization problem we asked in the beginning of this section. Recall that the plucking polynomial of any rooted tree can be written as the product of $q$-binomial coefficients.

\begin{theorem}\label{Theorem 5.6}
Consider a product of $q$-binomial coefficients:
$$P(q) = P_1(q)P_2(q)\cdots P_k(q), \mbox{ }  \mbox{ where } P_i(q) = {m_i+n_i \choose m_i,n_i}_q.$$
Then the product can be realized as $Q(T)$ for some rooted tree $T$ if and only if the numerator of $P(q)$ does not repeat any $q$-integer.
\end{theorem}

\begin{proof}
According to Theorem 5.3, it suffices to show that if the numerator of $P(q)$ does not repeat any $q$-integer, then there exists a rooted tree $T$ such that $Q(T)=P(q)$. Without loss of generality, we assume that $m_1+n_1\geq m_2+n_2\geq \cdots \geq m_k+n_k$. Denote
\begin{center}
$A=\{m_2+n_2, \cdots, m_k+n_k\}$ and $B=\{m_1, n_1, \cdots, m_k, n_k\}$.
\end{center}
Namely, $A$ is a set which contains $k-1$ integers and $B$ is a set which consists of $2k$ integers. If the numerator of $P(q)$ does not repeat any $q$-integer, we claim that there exists an injection $f$ from $A$ to $B$ such that $f(m_i+n_i)\geq m_i+n_i$, for all $2\leq i\leq k$. In fact we can define $f$ beginning with $m_2+n_2$, if $m_1\geq m_2+n_2$, then we define $f(m_2+n_2)=m_1$. Next let us consider $m_3+n_3$, if any one of $\{n_1, m_2, n_2\}$, say, $n_2$, is greater than or equal to $m_3+n_3$, then we define $f(m_3+n_3)=n_2$. The key point is this process will not stop until all elements of $A$ have been assigned an image under $f$. If not, we assume that after defining $f(m_2+n_2), \cdots, f(m_{i-1}+n_{i-1})$ we can not find the image of $m_i+n_i$ in $B-\{f(m_2+n_2), \cdots, f(m_{i-1}+n_{i-1})\}$. Note that $m_1+n_1\geq m_i+n_i$, hence $[m_1+n_1]_q!$ contains the $q$-integer $[m_i+n_i]_q$. Now since no integer in $B-\{f(m_2+n_2), \cdots, f(m_{i-1}+n_{i-1})\}$ is greater than or equal to $m_i+n_i$, it follows that the $q$-integer $[m_i+n_i]_q$ appears at least twice in the the numerator of $P(q)$. This contradicts with the assumption.

Now we explain how to construct a binary rooted tree $T$ with $Q(T)=P(q)$.  First we regard $T$ as the wedge product of $T_1$ and $T_2$, where $|E(T_1)|=m_1$ and $|E(T_2)|=n_1$. If $f(m_2+n_2)=m_1$, then $T_1$ can be described as the wedge product of $T_3$ and $T_4$ with a string consists of $m_1-(m_2+n_2)$ edges, where $|E(T_3)|=m_2$ and $|E(T_4)|=n_2$. With the help of $f$, we can construct $T$ step by step. In particular, if some integer, say $n_1$, is not the image of any element in $A$ under $f$, then $T_2$ is a straight line with $n_1$ edges. In this way finally we can obtain a binary tree $T$ such that $Q(T)=P(q)$. See Figure $5.3$ for a simple example.
\end{proof}

\begin{figure}[h]
\centerline{\psfig{figure=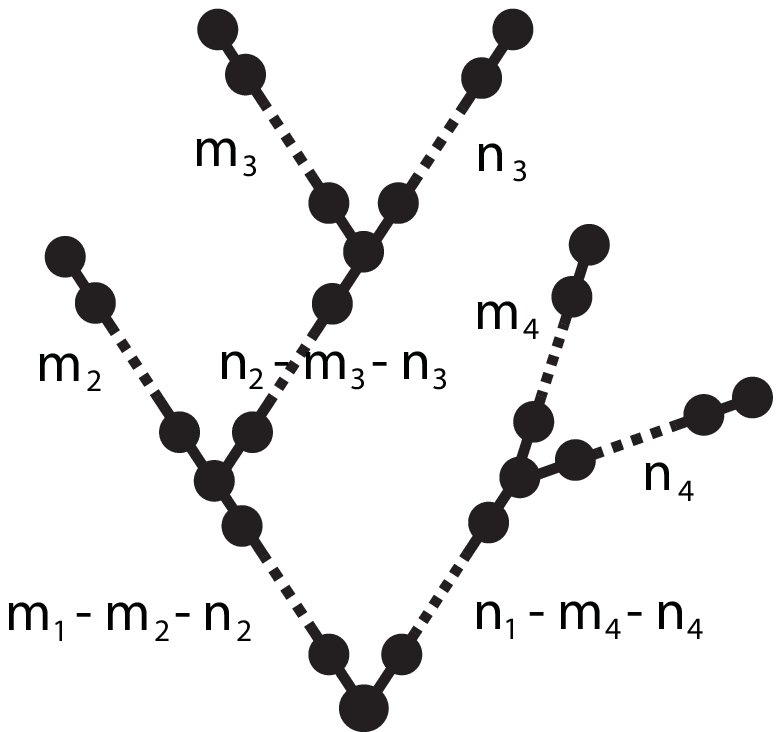,height=8cm}}\ \\ \ \\
\centerline{Figure $5.3$: The realization of $P(q)=\prod\limits_{i=1}^4{m_i+n_i \choose m_i, n_i}_q$ with a given $f$}
\centerline{satisfying $f(m_{2}+n_{2})=m_{1},f(m_{3}+n_{3})=n_{2}, f(m_{4}+n_{4})=n_{1}$}
\end{figure}

\subsection{Uniqueness}

We have an infinite number of different rooted trees with the same reduced form (i.e. string can be of any length) therefore with the same plucking polynomial. However, for reduced rooted trees we have:

\begin{proposition}
For reduced rooted trees, the function $T \mapsto Q(T)$ is finite-to-one.
\end{proposition}

\begin{proof}
By Theorem \ref{Theorem 5.3}$(2)$ the polynomial $Q(T)$ determines the number of edges of a reduced rooted tree $T.$ Because we can have only a finite number of trees with given number of edges, the proposition follows.
\end{proof}

There are many examples of different reduced rooted trees of the same plucking polynomial.
The simplest one is of five edges as in Figure $1.1$ (the trees differ by changing a root\footnote{For a balanced tree $T={\psfig{figure=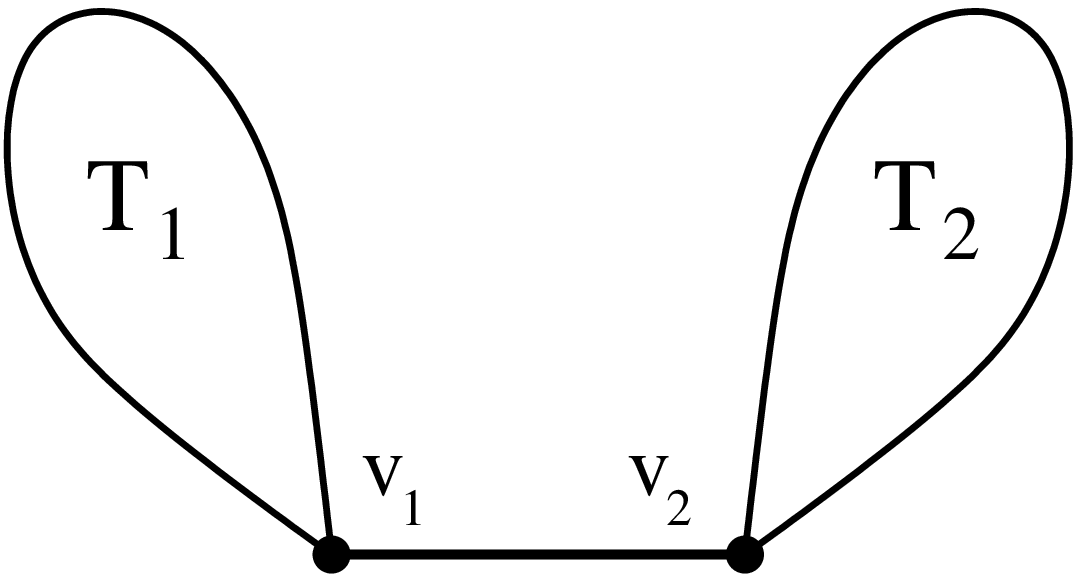,height=1.0cm}}$ where $|E(T_1)|=|E(T_2)|$ we can move the root from $v_1$ to $v_2$ without changing the
plucking polynomial $Q(T)$, compare \cite{Prz-3}.}). Furthermore, for any $n,$ we can construct $n$ different trees with the same plucking polynomial:

\begin{example}
The product of $q$-integers $([4]_q [5]_q) ([7]_q [8]_q) \cdots ([4+3(n-2)]_q [5+3(n-2)]_q)$ can be realized as $n$ different rooted trees (See Figure $5.4$ for the case of $n=5$).
\end{example}

\begin{figure}[h]
\centerline{\psfig{figure=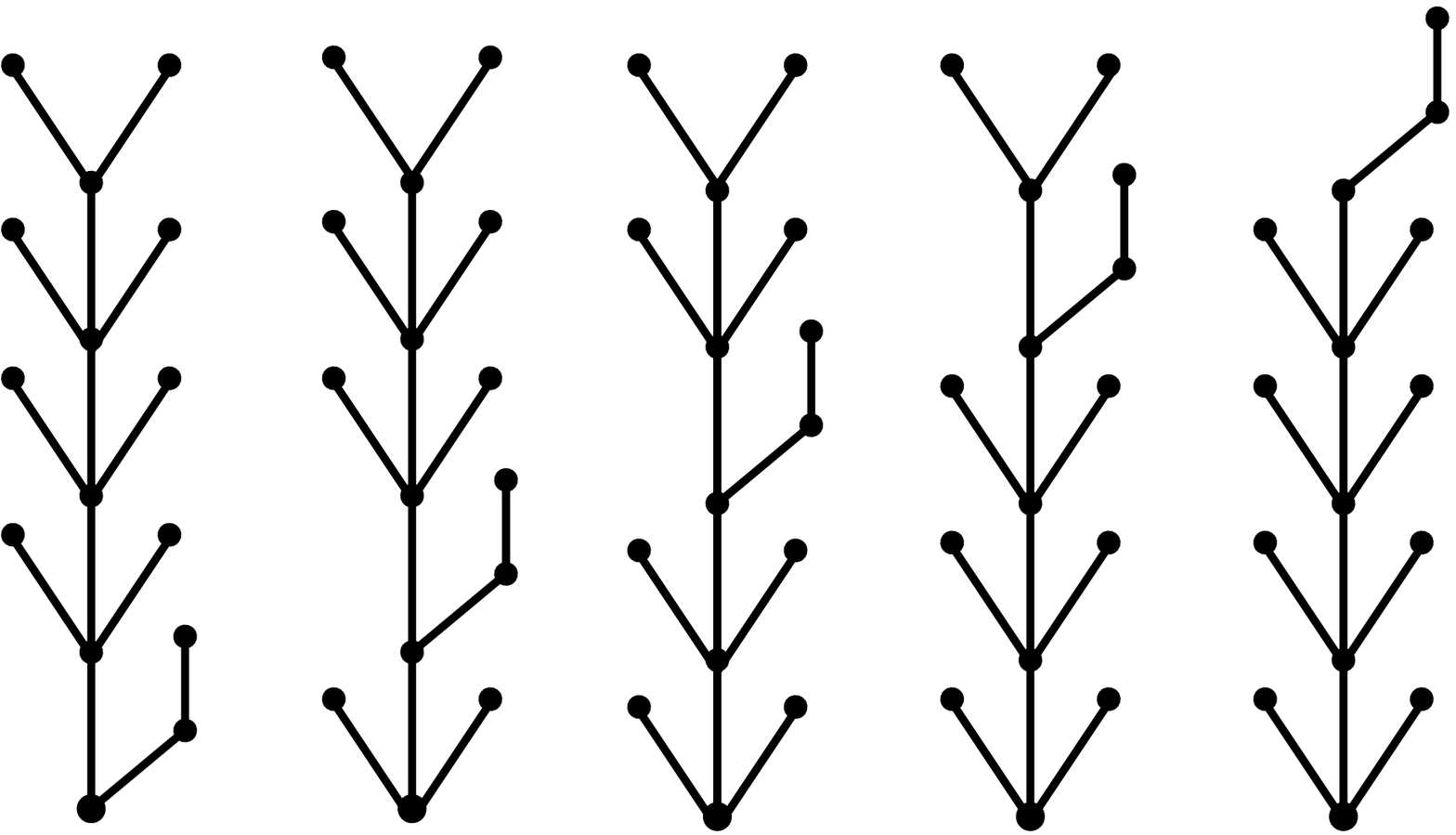,height=5.2cm}}\ \\
\centerline{Figure $5.4$: The realizations of $[4]_q[5]_q[7]_q[8]_q[10]_q[11]_q[13]_q[14]_q$}
\end{figure}

As we have seen in Figure $5.4$, in general the realization of $[a_1]_q[a_2]_q\cdots[a_k]_q$ is not unique.
However, it is easy to conclude that for one $q$-integer $[a_1]_q$, the realization is unique. For $[a_1]_q[a_2]_q$, the realization
is unique if and only if $(i)$ $a_2\geq a_1+2,$ $(ii)$ $a_1=2, a_2=3,$ or $(iii)$ $a_1=3, a_2=4$. When $a_2=a_1+1\geq5$, there are only two distinct
realizations for $[a_1]_q[a_2]_q$, and these two distinct rooted trees are isomorphic as (un-rooted) trees. More generally, we have the following result.

\begin{proposition}\label{Proposition 5.8}

Assume that $a_1\leq a_2\leq\cdots\leq a_k$, if $a_{i+1}-a_i\geq2$, then there exists only one rooted tree $T$ with $Q(T)=[a_1]_q[a_2]_q\cdots[a_k]_q$.

\end{proposition}

\begin{proof}
First we claim that if we write $T$ in the form of a wedge product $T_1\vee\cdots\vee T_n$, then $n=2$ and $|E(T_1)|=1$. By Theorem \ref{Theorem 5.3}$(2)$ the numerator of the reduced form of $Q(T)$ contains $[a_k]_q$ and $[a_k-1]_q.$ This contradicts the assumption that $[a_k]_q - [a_k-1]_q \geq 2.$ Now let us consider $T_2$, beginning with the root. Let $v$ be the next vertex with degree $\geq 3$ (which exists if $k \geq 2$). Due to a similar reason, we have $T^v=T^v_1\vee T^v_2$ and $|E(T^v_1)|=1$. Continuing the discussion we will obtain the unique realization of $[a_1]_q[a_2]_q\cdots[a_k]_q$.

\end{proof}

We remark that the converse of Proposition \ref{Proposition 5.8} does not hold in general. For example, the realization of  $[2]_q[3]_q$ or $[2]_q[4]_q[5]_q$ is unique.

We end this paper with a problem.

\begin{problem}\label{Problem 5.11}
Find a family of moves on rooted trees that preserve plucking polynomial such that if two trees have the same plucking polynomial,
then they are related by a finite number of moves.
\end{problem}

An example of an elementary move is illustrated in Figure $5.5$ where we modify a tree by exchanging $T_{1}$ and $T_{2}$ where $|E(T_{1})|=|E(T_{2})|.$ Maybe this move solves the problem above.

\begin{figure}[h]
\centerline{\psfig{figure=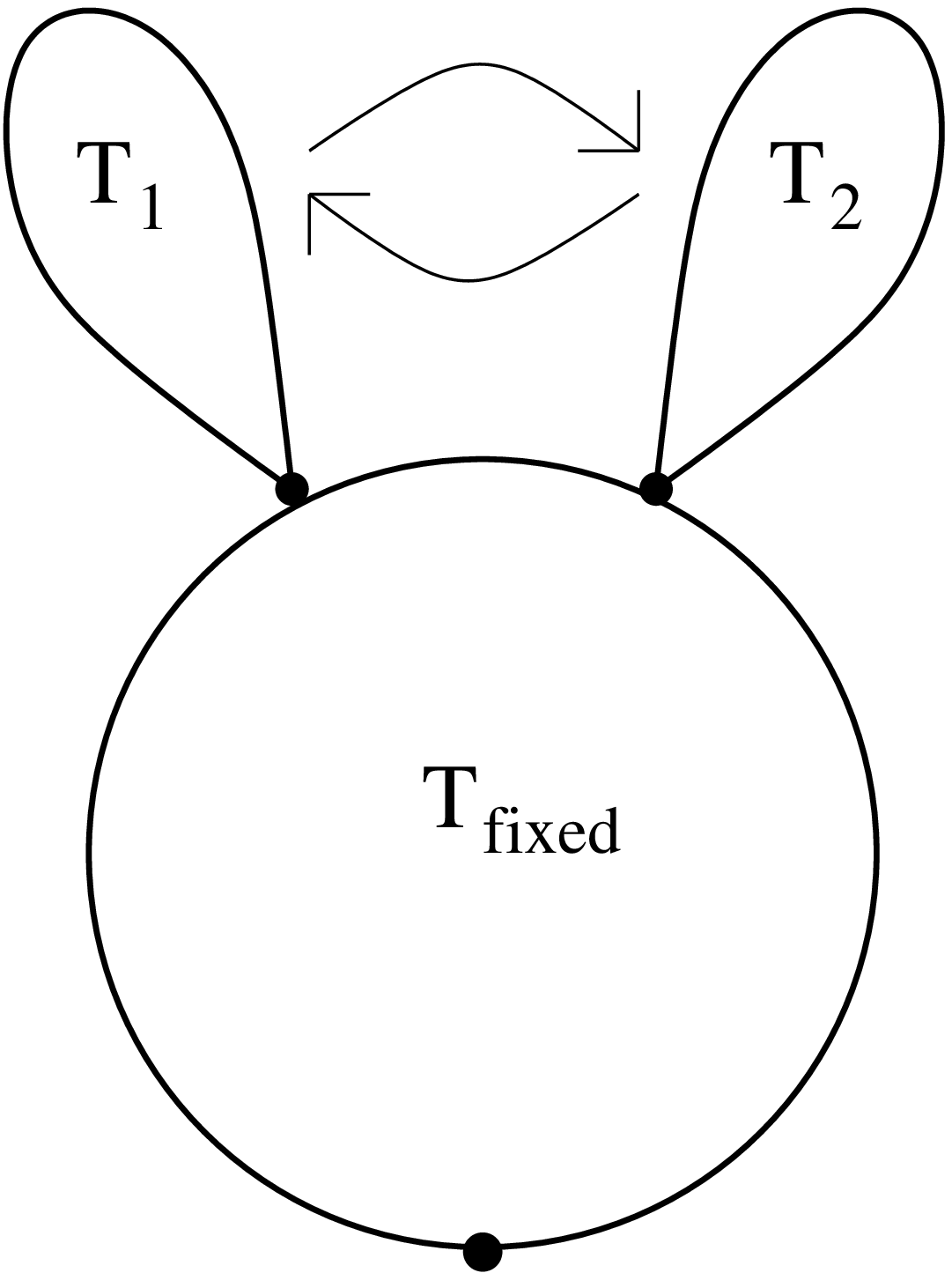,height=5.2cm}}\ \\
\centerline{Figure $5.5$: An exchange move}
\end{figure}

\section{Acknowledgements}\ \\
We would like to thank Jeremy Siegert for helping the authors edit this paper.\\
Z.~Cheng was supported by NSFC 11301028 and NSFC 11571038.\\
S.~Mukherjee was supported by Presidential Merit Fellowship of GWU.\\
J.~H.~Przytycki was partially supported by the  Simons Foundation Collaboration Grant for Mathematicians--316446.\\
X.~Wang was supported by the GWU fellowship.\\
S.~Y.~Yang was supported by the GWU fellowship.

 \end{document}